\documentclass[12pt,a4paper]{amsart}

\usepackage{multicol,adjmulticol}
\usepackage{mathtools}
\usepackage{float}
\usepackage{amssymb, amsfonts, amsmath, verbatim, graphicx,enumerate}
\usepackage{todonotes}
\presetkeys{todonotes}{fancyline, color=gray!30}{} 

\AtBeginDocument{%
   \def\MR#1{}
}

\newcommand{\g}{\mathfrak{g}}
\newcommand{\h}{\mathfrak{h}}
\newcommand{\m}{\mathfrak{m}}
\newcommand{\Span}{\mathrm{Span}}
\newcommand{\ad}{\mathrm{ad}}

\def\z2{\mathbb{Z}_2}
\def\^{\wedge}
\def\w{\omega}
\def\F{\mathbb{F}}

\usepackage{amsthm}

\makeatletter
\newtheorem*{rep@theorem}{\rep@title}
\newcommand{\newreptheorem}[2]{%
\newenvironment{rep#1}[1]{%
 \def\rep@title{#2 \ref{##1}}%
 \begin{rep@theorem}}%
 {\end{rep@theorem}}}
\makeatother

\theoremstyle{plain}
\newtheorem{theorem}{Theorem}

\newreptheorem{theorem}{Theorem}

\newtheorem{proposition}{Proposition}
\newtheorem{lemma}{Lemma}

\newtheorem*{theorem*}{Theorem}
\newtheorem*{corollary*}{Corollary}

\theoremstyle{definition}
\newtheorem{definition}{Definition}
\newtheorem*{definition*}{Definition}

\theoremstyle{remark}

\newtheorem*{example*}{Example}

\newtheorem*{exercise*}{Exercise}

\begin{document}
\title[Filiform Lie Algebras]{On the Betti numbers of filiform Lie algebras over fields of characteristic two}

\author{Ioannis Tsartsaflis}
\address{Department of Mathematics and Statistics, La Trobe University, Melbourne, Australia
3086}
\email{J.Tsartsaflis@latrobe.edu.au, John@purebyteslab.com}

\begin{abstract}
An $n$-dimensional Lie algebra $\g$ over a field $\F$ of characteristic two is said to be of \emph{Vergne type} if there is a basis $e_1,\dots,e_n$ such that $[e_1,e_i]=e_{i+1}$ for all $2\leq i \leq n-1$ and $[e_i,e_j] = c_{i,j}e_{i+j}$ for some $c_{i,j} \in \F$ for all $i,j \ge 2$ with $i+j \le n$. 
We define the algebra $\mathfrak{m}_0$ by its nontrivial bracket relations: $[e_1,e_i]=e_{i+1}, 2\leq i \leq n-1$, and the algebra 
$\mathfrak{m}_2$: $[e_1, e_i ]=e_{i+1}, 2 \le i \le n-1$,
$[e_2, e_j ]=e_{j+2}, 3 \le j \le n-2$.

We show that, in contrast to the corresponding real and complex cases, $\m_0(n)$ and $\m_2(n)$ have the same Betti numbers. We also prove that for any Lie algebra of Vergne type of dimension at least $5$, there exists a non-isomorphic algebra of Vergne type with the same Betti numbers.
\end{abstract}

\maketitle

\section{Introduction}
Throughout this paper, $\F$ denotes a field of characteristic two and all cohomologies are taken with trivial coefficients. In this paper, we study a particular class of nilpotent Lie algebras over $\F$.

\begin{definition}\label{def:filiform}
Let $\g$ be an $n$-dimensional nilpotent Lie algebra over $\F$. If there is an element $x \in \g$ such that $\ad^{n-2}(x) \ne 0$, then we say that  $\g$ is \emph{filiform}.
\end{definition}

\begin{definition}\label{def:vergnetype}
Let $\g$ be an $n$-dimensional filiform Lie algebra over $\F$. We say that $\g$ is of \emph{Vergne type} if there is a basis $e_1,\dots,e_n$ for $\g$ such that $[e_1,e_i]=e_{i+1}$ for all $2\leq i \leq n-1$ and $[e_i,e_j] = c_{i,j}e_{i+j}$ for some $c_{i,j} \in \F$ for all $i,j \ge 2$ with $i+j \le n$. If $\g$ is of Vergne type, we call such a basis a \emph{Vergne basis}.
\end{definition}
Lie algebras of Vergne type are filiform Lie algebras; if $e_1,\dots,e_n$ is a Vergne basis for a Lie algebra $\g$, then $e_1$ is an element of maximal rank.

We define $\m_0(n)=\Span(e_1,\dots,e_n)$ to be the Lie algebra with nonzero bracket relations $[e_1,e_i]=e_{i+1}$ for $2\leq i \leq n-1$, and $\m_2(n)=\Span(e_1,\dots,e_n)$ the Lie algebra with nonzero bracket relations $[e_1,e_i]=e_{i+1},$ for $2 \le i \le n-1$ and $[e_2,e_j]=e_{j+2}$ for $3 \le j \le n-2$. We keep the notation $\m_0$ and $\m_2$ for these algebras as introduced in the classification of A. Fialowski in \cite{Fial}.

Lie algebras of Vergne type are Lie algebras of maximal class, in the sense that they have maximal nilpotency. Infinite Lie algebras of maximal class over fields of characteristic two have been studied in \cite{max1,max2,max3}. The cohomology of Lie algebras of maximal class have been studied extensively over a field of characteristic zero \cite{D, Fial, FM, Vergne}.
In \cite{FM} the authors give a full description of the cohomology with trivial coefficients of the infinite analogues of $\m_0(n)$ and $\m_2(n)$. The finite dimensional case is also discussed \cite{FM}, for which we know the full description of the cohomology of $\m_0(n)$. However, for $\m_2(n)$ our knowledge is limited to the first few Betti numbers and it is known \cite{million} that over a field of characteristic zero, these two Lie algebras have different Betti numbers in dimension $n \ge 5$. Over a field of characteristic two we have the following theorem:

\begin{theorem}\label{m0m2}
For $n \in \mathbb{N}$, the Lie algebras $\m_0(n)$ and $\m_2(n)$ over $\F$ have the same Betti numbers.
\end{theorem}
In Section \ref{sect:general}, we will prove the following theorem:
\begin{theorem}\label{t:second}
For any nilpotent Lie algebra of Vergne type of dimension at least $5$ over $\F$, there exists a non-isomorphic Lie algebra of Vergne type having the same Betti numbers.
\end{theorem}

Note that Definition \ref{def:filiform} is not the standard definition for filiform Lie algebras. The common usage of the term is to say that a nilpotent Lie algebra is filiform if it is of maximal nilpotency. These two definitions are equivalent over a field of characteristic zero as shown in \cite{Vergne}. Moreover, it is not hard to check that Vergne's proof works for any infinite field, even of finite characteristic. However, this is not the case for algebras over $\z2$ as there are counter-examples showing that maximal nilpotency does not imply Definition \ref{def:filiform}. 

Throughout this paper, when we define a Lie algebra by relations, we will omit the zero ones and also those which follow from the given ones by skew-symmetry.
\section{Preliminaries}
Let $\g$ be a Lie algebra of Vergne type over $\F$ with basis $e_1,\dots, e_n$, and let $e^1,\dots,e^n$ denote the corresponding dual basis of the dual space $\g^*$ of $\g$. One can define a grading in the space of $k$-forms 
$\Lambda^{k}(\mathfrak{\g}) = {\bigoplus} \Lambda^k_{m}(\mathfrak{\g})$, 
where the finite-dimensional subspace $\Lambda^k_{m}(\mathfrak{\g})$ is spanned by \emph{homogeneous elements} $e^{i_1} {\wedge} \dots {\wedge} e^{i_k}$ such that  $i_1<\dots <i_k$ and $i_1{+}\dots{+}i_k=m$. 
We define the \emph{degree} of a homogeneous element  $x \in \Lambda^k_{m}(\mathfrak{\g})$ to be $m$ and its \emph{topological degree} to be $k$.

The differential on homogeneous elements of topological degree $1$ coincides with the dual mapping of the Lie bracket
$[ \, , ]: \Lambda^2 \g \to \g$ and we extend it to elements of topological degree greater than $1$ by
$$
d(\rho \wedge \eta)=d(\rho) \^ \eta+ \rho \^ d(\eta).
$$
For an arbitrary $n$-dimensional Lie algebra of Vergne type $\g$ with relations $[e_1,e_i]=e_{i+1}$ for all $2\leq i \leq n-1$ and $[e_i,e_j] = c_{i,j}e_{i+j}$ for some $c_{i,j} \in \F$ for all $i,j \ge 2$ with $i+j \le n$, the differential is:
\[d(e^1)=d(e^2)=0, \quad d(e^k)= e^1\^e^{k-1} + \sum_{\substack{i+j=k\\1<i<j}} c_{i,j}e^i\^e^j, \text{ for } k\ge3.\]
For Lie algebras of Vergne type, the grading is compatible with the differential 
$d$ and with the exterior product:
$$d \Lambda^k_{m} (\mathfrak{\g})
\subset \Lambda^{k{+}1}_{m} (\mathfrak{\g}), \qquad
\Lambda^{q}_{m}(\mathfrak{\g}) \wedge
\Lambda^{p}_{l}(\mathfrak{\g}) \subset
\Lambda^{q{+}p}_{m+l}(\mathfrak{\g}).
$$
Hence, every cocycle is a sum of homogeneous elements and we can write
\[H^k(\g) = \bigoplus_{m=k(k+1)/2}^{kn -  k(k-1)/2}H^k_m(\g)\]
where $H^k_m(\g)$ is spanned by the cohomology classes of homogeneous $k$-cocycles of degree $m$
\[H^k_m(\g) = \{ \bar{\w} \in H^k(\g) : \w =\sum_{\substack{1 \leq  i_1 <\dots < i_k \leq n \\ i_1 + \dots + i_k =m}} \w_{i_1,\dots,i_k} e^{i_1}\^ \dots \^ e^{i_k},\; \w_{i_1,\dots,i_k} \in \F \},\]
where $\bar{\w}$ is the cohomology class of $\w$.
Let us note here that the first cohomology $H^1$ of an $n$-dimensional Lie algebra of Vergne type $\g$ has dimension $2$; if we choose a Vergne basis for $\g$, then $H^1(\g) = \Span(\overline{e^1},\overline{e^2})$.

\section{The special case}\label{sect:special}

Consider the $n$-dimensional vector space $V=\Span(e_1,\dots,e_n)$ over $\F$. We define the algebra $\m_0(n)$ to be the vector space $V$ with relations $[e_1,e_i]=e_{i+1}$ for $i \geq 2$, and the algebra $\m_2(n)$ to be the same vector space $V$ with relations $[e_1,e_i]=e_{i+1},$ for $i \geq 2$ and $[e_2,e_j]=e_{j+2}$ for $j \geq 3$.

We define the linear map $D_1:\Lambda^1(e^1,\dots,e^n) \to \Lambda^1(e^1,\dots,e^n)$ by
\[D_1(e^1)=D_1(e^2)=0, \; \;D_1(e^i)=e^{i-1},\; i \ge 3,\]
and the linear map $D_2:\Lambda^1(e^1,\dots,e^n) \to \Lambda^1(e^1,\dots,e^n)$ by
\[D_2(e^i)=0,i \leq 4,\;\;D_2(e^i)=e^{i-2}, \;i \ge 5.\]
We extend the action of these maps to the linear maps $D_i :\Lambda^*(e^1,\dots,e^n) \to \Lambda^*(e^1,\dots,e^n)$ as derivations defined by $D_i(\xi\^\zeta) = D_i(\xi) \^ \zeta + \xi\^D_i(\zeta)$ for $i \in \{1,2\}$.
The differential $d_0$ on $\m_0(n)$ is as follows:
\[d_0(e^1)=d_0(e^2)=0, \quad d_0(e^i)=e^1\^e^{i-1}, \text{ for all } 1\leq i \leq n,\]
and the differential $d_2$ on $\m_2(n)$ is:
\[d_2(e^1)=d_2(e^2)=0, \quad d_2(e^i)=\begin{cases}e^1\^e^{i-1}& \mbox{if } i = 3,4,\\ 
e^1\^e^{i-1} + e^2\^e^{i-2}& \mbox{if } i \ge 5. \end{cases} \]
We notice here that the operator $d_0=e^1\^D_1$ is the differential on $\m_0(n)$ and the operator $d_2=e^1\^D_1 + e^2\^D_2$ is the differential on $\m_2(n)$. 

Now we are going to make some remarks on the similarity between $D_2$ and $D_1^2$. First, we observe that $D_1^2$ is a derivation. Indeed, we have
\[D_1^2(\xi\^\zeta) = D_1(D_1(\xi) \^ \zeta + \xi\^D_1(\zeta)) = D_1^2(\xi) \^ \zeta + \xi\^D_1^2(\zeta).\]
Furthermore, it follows from the definitions that $D_2$ acts on monomials of topological degree one as follows:
\[ D_2(e^k)=\begin{cases}D_1^2(e^k) + e^2& \mbox{if } k = 4,\\ 
D_1^2(e^k) & \mbox{otherwise}. \end{cases} \]
If we take an arbitrary form $\w$, we can write it as $\w = e^4\^\w' + \w''$ where $\w'$ and $\w''$ have no component in the $e^4$ direction, and then $e^2\^D_2(\w) = e^2\^(e^4\^D_2(\w') + D_2(\w''))$ and also $e^2\^D_1^2(\w) = e^2\^(e^2\^\w' + e^4\^D_1^2(\w') + D_1^2(\w''))=e^2\^(e^4\^D_1^2(\w') + D_1^2(\w''))$. The following lemma is immediate.

\begin{lemma}\label{d1dstar}
For all $\w \in \Lambda^{*}(e^1,\dots,e^n)$, we have $e^2\^D_2(\w) =e^2\^D_1^2(\w) $.
\end{lemma}

\begin{definition}\label{def:involution}
Let $n,k \in \mathbb{N}$ with $2\leq k \leq n$. We define the map $f^n_k:\Lambda^k(e^1,\dots,e^n) \to \Lambda^k(e^1,\dots,e^n)$ by $f^n_k(e^1\^x + e^2\^y + z) = e^1\^x + e^2\^(y + D_1(x)) + z$, where $x \in \Lambda^{k-1}(e^2,\dots,e^n), y \in \Lambda^{k-1}(e^3,\dots,e^n)$ and  $z \in \Lambda^{k}(e^3,\dots,e^n)$.
\end{definition}

For brevity, we will refer to $f^n_k$ just by $f$ unless a clarification is needed.
Note that $f$ is an involution; indeed, $f(f(e^1\^x + e^2\^y + z)) = e^1\^x + e^2\^(y + D_1(x) + D_1(x)) + z = e^1\^x + e^2\^y + z$.

\begin{proposition}\label{p:m2m0}
Let $n\ge 5$ and let $d_0,d_2$ be the differentials of $\m_0(n), \m_2(n)$ respectively. Then for all $k\ge 2$, the following diagram commutes
\begin{center}
\begin{tikzpicture}[every node/.style={midway}]
\matrix[column sep={10em,between origins},
        row sep={3em}] at (0,0)
{ \node(g1k)   {$\Lambda^k(\m_0(n))$}  ; & \node(g2k) {$\Lambda^k(\m_2(n))$}; \\
  \node(g1k1)   {$\Lambda^{k+1}(\m_0(n))$}  ; & \node(g2k1) {$\Lambda^{k+1}(\m_2(n))$};            \\};
\draw[->] (g1k) -- (g2k) node[anchor=south]  {$f$};
\draw[->] (g1k1) -- (g2k1) node[anchor=south]  {$f$};
\draw[->] (g1k) -- (g1k1) node[anchor=east]  {$d_0$};
\draw[->] (g2k) -- (g2k1) node[anchor=east]  {$d_2$};
\end{tikzpicture}
\end{center}
\end{proposition}
\begin{proof}
Let $h=e^1\^x + e^2\^y + z \in \Lambda^{k}(e^1,\dots,e^n)$ with $x \in  \Lambda^{k-1}(e^2,\dots,e^n), y \in  \Lambda^{k-1}(e^3,\dots,e^n)$ and $z \in  \Lambda^{k}(e^3,\dots,e^n)$. We calculate
\[d_2(h)=e^1\^e^2\^D_2(x) + e^1\^e^2\^D_1(y) + e^1\^D_1(z) + e^2\^D_2(z)\]
and
\[f(d_2(h))=e^1\^e^2\^D_2(x) + e^1\^e^2\^D_1(y) + e^1\^D_1(z) + e^2\^D_2(z) + e^2\^D_1^2(z)\]
which, by Lemma \ref{d1dstar}, gives
\[f(d_2(h))=e^1\^e^2\^D_2(x) + e^1\^e^2\^D_1(y) + e^1\^D_1(z).\]
We also have:
\begin{align*}
d_0(f(h)) &= d_0(e^1\^x + e^2\^(y+D_1(x)) + z) 
\\&=e^1\^e^2\^D_2(x) + e^1\^e^2\^D_1(y) + e^1\^D_1(z),
\end{align*}
and the claim follows.
\end{proof}

Theorem \ref{m0m2} is an immediate application of Proposition \ref{p:m2m0}.
\begin{proof}[Proof of Theorem \ref{m0m2}]
The $k$-th Betti number for an $n$-dimensional Lie algebra $\g$ is given by
\[b_k = \dim(Z_k(\g)) + \dim(Z_{k-1}(\g)) - \binom{n}{k-1},\]
where $Z_k(\g)$ denotes the space of $k$-cocycles. From Proposition  \ref{p:m2m0} and the fact that $\dim(Z_1(\m_0(n)))= \dim(Z_1(\m_2(n)) )=2$, it follows that  $\dim(Z_k(\m_0(n)))= \dim(Z_k(\m_2(n)) )$ for $1 \leq k \leq n$.
\end{proof}
The explicit computation of these Betti numbers is much harder than in the case of zero characteristic. We have already seen that $b_1=2$, and one can show that $b_2=\lfloor\frac12(n+1)\rfloor$. The computation of $b_3$ is already quite technically involved and will be presented in a forthcoming paper.

\section{The general case}\label{sect:general}

A Lie algebra, in general, has several central extensions as described in \cite{neeb} or \cite{weibel}. In our next lemma we give a condition for a central extension of a Lie algebra of Vergne type to be also of Vergne type. 

Given a Lie algebra of Vernge type $\g$ with Vergne basis $e_1,\dots, e_n$ and a $2$-cocycle $\w$ of $\g$, we define the central extension
$\g(\w) =\g\oplus \F$ of $\g$ to be the Lie algebra with relations
\begin{equation}\label{def:omega}
[e_i,e_j]_{\g(\w)} = [e_i,e_j]_{\g} + \w(e_i,e_j)e_{n+1},
\end{equation}
where $e_{n+1}$ is a basis for $\F$.

\begin{lemma}\label{l:coolextensions}
{\ }
\begin{enumerate}[(a)]
\item{Let $\g$ be a Lie algebra of Vergne type with Vergne basis $e_1,\dots, e_n$ and let $\w$ be a homogeneous $2$-cocycle with a nonzero $e^1\^e^n$ component. Then $\g(\w)$ is a Lie algebra of Vergne type.}
\item{Let $\g'$ be a Lie algebra of Vergne type with Vergne basis $e_1,\dots, e_{n+1}$. For the Lie algebra of Vergne type $\g=\g'/\Span(e_{n+1})$, which we identify with $\Span(e_1,\dots,e_{n})$, there exists a homogeneous $2$-cocycle with a nonzero $e^1\^e^n$ component $\w$, such that $\g'=\g(\w)$.}
\end{enumerate}
\end{lemma}
\begin{proof}
(a) This follows from the definition of an extension and the choice of $\w$. 
\newline
(b) Consider the homogeneous $2$-form $\w$ defined by $\w = d(e^{n+1})$. 
The projection of $\w$ to $\Lambda^2(\g)$ is a homogeneous $2$-cocycle with a nonzero $e^1\^e^n$ component.
\end{proof} 

Let $\g$ be a Lie algebra of Vergne type with Vergne basis $e_1,\dots,e_n$ and bracket relations $[e_1,e_k] = e_{k+1}$ and $[e_i,e_j]=c_{i,j}e_{i+j}$, for some $c_{i,j} \in \F$ where $k \ge 2$ and $i,j \ge 2$.
As in the previous section, we define the map $D_1:\Lambda^1(e^1,\dots,e^n) \to \Lambda^1(e^2,\dots,e^n)$ by
\[D_1(e^1)=D_1(e^2)=0, \; \;D_1(e^i)=e^{i-1}, i \ge 3,\]
and we extend this map to $D_1 :\Lambda^*(e^1,\dots,e^n) \to \Lambda^*(e^2,\dots,e^n)$ as a derivation by $D_1(\xi\^\zeta) = D_1(\xi) \^ \zeta + \xi\^D_1(\zeta)$.

Next, we define the operator $R=e^1\^D_1 + d$, where $d$ is the differential of $\g$. Observe that $R$ is a derivation of the exterior algebra as it is the sum of two derivations. Moreover, we have that $R(e^1)=R(e^2)=R(e^3)=R(e^4)=0$ and 
\[R(e^5)=c_{2,3} e^2\^e^3, \dots, R(e^n) = \sum_{\substack{i+j=n\\1<i<j}} c_{i,j}e^i\^e^j.\]
Finally, we observe that the image of $R$ lies in $\Lambda^*(e^2,\dots,e^n)$.

\begin{lemma}\label{lemma:D1Rvergne}
Let $\g$ be a Lie algebra of Vergne type with Vergne basis $e_1,\dots, e_n$ and let $e^1\^x + e^2\^y + z \in \Lambda^k(\g^*)$, where $x \in \Lambda^{k-1}(e^2,\dots,e^n), y \in \Lambda^{k-1}(e^3,\dots,e^n), z \in \Lambda^{k}(e^3,\dots,e^n)$ and $k \ge 2$. If $e^1\^x + e^2\^y + z$ is a cocycle, then
\[e^2\^D_1(z) = e^2\^R(x).\]
\end{lemma}
\begin{proof}
If we assume that $e^1\^x + e^2\^y + z$ is a cocycle, then
\begin{align*}
0&=d(e^1\^x + e^2\^y + z)
\\&=e^1\^R(x) + e^1\^e^2\^D_1(y) + e^2\^R(y) + e^1\^D_1(z) + R(z)
\end{align*}
Hence, we have $e^2\^R(y) + R(z) =0$ and $e^1\^(R(x) + e^2\^D_1(y) + D_1(z))=0$. Taking the exterior product of the second equation with $e^2$, we obtain the claim.
\end{proof}

\begin{proposition}\label{p:extensions}
Let $\g_1, \g_2$ be $n$-dimensional Lie algebras of Vernge type with Vergne bases $e_1,\dots, e_n$. Let $\w$ be a homogeneous $2$-cocycle of $\g_1$ with a nonzero $e^1\^e^n$ component and denote by $f$ the involution from Definition \ref{def:involution}. If for $k \ge 2$ the diagram
\begin{center}
\begin{tikzpicture}[every node/.style={midway}]
\matrix[column sep={10em,between origins},
        row sep={3em}] at (0,0)
{ \node(g1k)   {$\Lambda^k(\g_1)$}  ; & \node(g2k) {$\Lambda^k(\g_2)$}; \\
  \node(g1k1)   {$\Lambda^{k+1}(\g_1)$}  ; & \node(g2k1) {$\Lambda^{k+1}(\g_2)$};            \\};
\draw[->] (g1k) -- (g2k) node[anchor=south]  {$f$};
\draw[->] (g1k1) -- (g2k1) node[anchor=south]  {$f$};
\draw[->] (g1k) -- (g1k1) node[anchor=east]  {$d_1$};
\draw[->] (g2k) -- (g2k1) node[anchor=east]  {$d_2$};
\end{tikzpicture}
\end{center}
commutes, then the following diagram also commutes
\begin{center}
\begin{tikzpicture}[every node/.style={midway}]
\matrix[column sep={10em,between origins},
        row sep={3em}] at (0,0)
{ \node(g1k)   {$\Lambda^k(\g_1(\w))$}  ; & \node(g2k) {$\Lambda^k(\g_2(f^n_2(\w))$}; \\
  \node(g1k1)   {$\Lambda^{k+1}(\g_1(\w))$}  ; & \node(g2k1) {$\Lambda^{k+1}(\g_2(f^n_2(\w))$,};            \\};
\draw[->] (g1k) -- (g2k) node[anchor=south]  {$f$};
\draw[->] (g1k1) -- (g2k1) node[anchor=south]  {$f$};
\draw[->] (g1k) -- (g1k1) node[anchor=east]  {$\hat{d_1}$};
\draw[->] (g2k) -- (g2k1) node[anchor=east]  {$\hat{d_2}$};
\end{tikzpicture}
\end{center}
where $d_i, \hat{d_i}$ denote the corresponding differential.
\end{proposition}
\begin{proof}
First, we observe that the degree of a homogeneous $k$-form is preserved by $f$. The assumption is that $d_1f=fd_2$ where $d_1$ and $d_2$ are the differentials of $\g_1$ and $\g_2$ respectively.
We are going to prove that the differential $\hat{d_1}$ of $\g_1(\w)$ and the differential $\hat{d_2}$ of $\g_2(f(\w))$ are conjugate by $f$.

Let $h=e^{n+1}\^x + y \in Z_k(\g_1(\w))$, where $x \in  \Lambda^{k-1}(e^1,\dots,e^n)$ and $y \in \Lambda^{k}(e^1,\dots,e^n)$. Furthermore, let  $x=e^1\^x_1 + e^2\^x_2 + x_3$ and $y=e^1\^y_1 + e^2\^y_2 + y_3$, where $x_1 \in  \Lambda^{k-2}(e^2,\dots,e^n), x_2 \in  \Lambda^{k-2}(e^3,\dots,e^n), x_3 \in  \Lambda^{k-1}(e^3,\dots,e^n)$ and $y_1 \in  \Lambda^{k-1}(e^2,\dots,e^n), y_2 \in  \Lambda^{k-1}(e^3,\dots,e^n), y_3 \in  \Lambda^{k}(e^3,\dots,e^n)$. The cocycle $\w$ has degree $n+1$ and has nonzero $e^1\^e^n$ component, hence, we write $\w=e^1\^e^n + e^2\^\w_2 + \w_3$, where $\w_2 \in \Span({e^{n-1}})$ and $\w_3 \in   \Lambda^{2}_{n+1}(e^3,\dots,e^n)$.

The image of $h$ is
\begin{align*}
f(h) &= f(e^1\^(x_1\^e^{n+1} + y_1) + e^2\^(x_2\^e^{n+1} + y_2) + x_3\^e^{n+1} + y_3)
\\&=e^1\^(x_1\^e^{n+1} + y_1)  + e^2\^(x_2\^e^{n+1} + y_2 + D_1(x_1\^e^{n+1} + y_1)) 
\\&\quad + x_3\^e^{n+1} + y_3
\\&=e^{n+1}\^f(x) + f(y) + e^2\^x_1\^e^n.
\end{align*}
The differential of $h$ is 
\[\hat{d_1}(h) =e^{n+1}\^d_1(x) + d_1(y) +  \w\^x \]
and the image of this is
\[
f(\hat{d_1}(h)) =e^{n+1}\^f(d_1(x)) + f(\w\^x) + f(d_1(y)) + e^2\^(R_1(x_1) + D_1(x_3))\^e^n.
\]
Using the assumption that $d_1$ and $d_2$ are conjugate by $f$ we can rewrite the last equation as follows:
\[
f(\hat{d_1}(h)) =e^{n+1}\^d_2(f(x)) + f(\w\^x) + d_2(f(y)) + e^2\^(R_1(x_1) + D_1(x_3))\^e^n.
\]

Now let us calculate the differential of the image of $h$.
\[
\hat{d_2}(f(h))=f(\w)\^f(x) + e^{n+1}\^d_2(f(x)) + d_2(f(y)) + e^2\^d_2(x_1\^e^n).
\]
If we manage to show that 
\begin{equation}\label{eq:dfg11}
f(\w)\^f(x) + f(\w\^x)  + e^2\^d_2(x_1\^e^n) + e^2\^(R_1(x_1) + D_1(x_3))\^e^n=0,
\end{equation}
then the proof is complete. We will break this expression into pieces and observe its validity.
In more details, we calculate 
\begin{align*}
f(\w)\^&f(x)=e^1\^(e^2\^(x_2\^e^n + D_1(x_1)\^e^n + x_1\^\w_2 + x_1\^e^{n-1}) + x_3\^e^n 
\\&+ x_1\^\w_3) +e^2\^(\w_2\^x_3 + x_2\^\w_3 + x_3\^e^{n-1} + D_1(x_1)\^\w_3) + x_3\^\w_3
\\f(\w&\^x)=e^1\^(e^2\^(x_2\^e^n + x_1\^\w_2) + x_3\^e^n + x_1\^\w_3)
\\&\qquad+e^2\^(\w_2\^x_3 + x_2\^\w_3 +  D_1(x_3\^e^n + x_1\^\w_3)) + x_3\^\w_3,
\end{align*}
and thus we have
\begin{align*}
f(\w)\^f(x) + f(\w\^x) &= e^1\^e^2\^D_1(x_1)\^e^n + e^1\^e^2\^x_1\^e^{n-1} 
\\&+ e^2\^D_1(x_3)\^e^n + e^2\^x_1\^D_1(\w_3).
\end{align*}
Using Lemma \ref{lemma:D1Rvergne} for the cocycle $\w$, we obtain
\begin{align*}
&e^2\^d_2(x_1\^e^n) =e^2\^f(d_1(f(x_1\^e^n)))=e^2\^f(d_1(x_1\^e^n))
\\&= e^2\^f(x_1\^d_1(e^n)) + e^2\^f(d_1(x_1)\^e^n)
\\&=e^1\^e^2\^x_1\^e^{n-1} + e^1\^e^2\^D_1(x_1)\^e^n 
+ e^2\^x_1\^R_1(e^n) + e^2\^R_1(x_1)\^e^n
\\&=e^1\^e^2\^x_1\^e^{n-1} + e^1\^e^2\^D_1(x_1)\^e^n 
+ e^2\^R_1(x_1)\^e^n + e^2\^x_1\^D_1(\w_3).
\end{align*}
Substituting the above expressions into the left-hand side of \eqref{eq:dfg11} we get
\begin{align*}
&e^1\^e^2\^D_1(x_1)\^e^n + e^1\^e^2\^x_1\^e^{n-1} + e^2\^D_1(x_3)\^e^n + e^2\^x_1\^D_1(\w_3)
\\&+ e^2\^R_1(x_1)\^e^n+e^1\^e^2\^x_1\^e^{n-1} + e^1\^e^2\^D_1(x_1)\^e^n 
\\&+ e^2\^x_1\^D_1(\w_3)+e^2\^R_1(x_1) \^e^n + e^2\^D_1(x_3)\^e^n =0,
\end{align*}
which is obviously true. Thus, we have $\hat{d_1}f=f\hat{d_2}$ and the diagram commutes.
\end{proof}

\begin{proof}[Proof of Theorem \ref{t:second}]
Let $\g$ be a Lie algebra of Vergne type of dimension $n$. The only nilpotent Lie algebras of Vergne type of dimension $5$ over any field are $\m_0(5)$ and $\m_2(5)$ according to \cite{graaf}, so if $n=5$, the required result follows from Theorem \ref{m0m2}. Assume that $n \ge 6$ and choose a Vergne basis $e_1, e_2,\dots,e_n$ for $\g$. 
Then, by Lemma \ref{l:coolextensions}, there exists an $(n-1)$-dimensional Lie algebra of Vergne type $\g^1$ and a homogeneous $2$-cocycle $\w_1$ of $\g^1$ with nonzero $e^1\^e^{n-1}$ component such that $\g = \g^1(\w_1)$. In the same way, there exists an $(n-2)$-dimensional Lie algebra of Vergne type $\g^2$ and a homogeneous $2$-cocycle $\w_2$ of $\g^2$ with nonzero $e^1\^e^{n-2}$ component such that $\g^1 = \g^2(\w_2)$. If we keep going backwards like this, we will reach a point where there exists a homogeneous $2$-cocycle $\w_{n-5}$ of a Lie algebra $\g^{n-5}$ with nonzero $e^1\^e^5$ component such that $\g^{n-6} = \g^{n-5}(\w_{n-5})$. This sequence of $2$-cocyles $(\w_1,\w_2,\dots,\w_{n-5})$ characterizes $\g$.

Now we construct a non-isomorphic algebra of Vergne type which has the same Betti numbers as $\g$. In order to do so, note that $\g^{n-5}$ is isomorphic to either $\m_0(5)$ or $\m_2(5)$. Let $\h^{n-5}$ be the other of the two algebras. We start with $\h^{n-5}$ and extend it by $f(\w_{n-5})$, where $f$ is the involution from Definition \ref{def:involution}. Denote $\h^{n-i} =\h^{n-i+1}(f(\w_{n-i+1}))$ for $i=5,\dots,n$. After applying $n-5$ times Proposition \ref{p:extensions}, we come to the conclusion that $\h^0$ and $\g$ have the same Betti numbers.

These algebras are not isomorphic, because the algebras that we started with, $\m_0(5)$ and $\m_2(5)$, are not isomorphic; this is due to the fact that $\m_0(5)$ has a codimension $1$ abelian ideal while there is no such an ideal for $\m_2(5)$.
\end{proof}
\section{Final remarks}
Figure \ref{tree} gives all the Lie algebras of Vergne type over $\z2$ up to dimension $12$. These computations were performed using \texttt{Maple}. The edges of this graph connect two algebras where the algebra on the bottom end is a central extension of the algebra on the upper end. The first number in the label of an algebra is its dimension and the second one is an increasing index.

Each of them is a $1$-dimensional central extension of a Lie algebra of Vergne type as described in Lemma \ref{l:coolextensions}. According to Theorem \ref{m0m2}, for any $n\ge 5$, we have that $\m_0(n)$ and $\m_2(n)$ have the same Betti numbers and the corresponding diagram commutes. Moreover, according to Proposition \ref{p:extensions} or Theorem \ref{t:second}, the diagram of $g(7,1)$ and $h(7,1)$ also commutes, hence, these algebras have the same Betti numbers; this is because they are central extensions of $\m_0(6)$ and $\m_2(6)$ for which the corresponding diagram commutes.

In general, to find the pairs of these algebras that have the same Betti numbers, we follow the same steps as in the proof of Theorem \ref{t:second}. For example, $\m_0(8)$ has two central extensions which are Lie algebras of Vergne type. If we choose the cocycle which has zero $e^2\^e^7$ component, then we pair this extension with the extension of $\m_2(8)$ with the cocycle of degree $9$ and nonzero $e^2\^e^7$ component; this follows from Proposition~2 and the definition of $f$. Based on this, we labelled them in such a way that $g(n,i)$ has the same Betti numbers as $h(n,i)$.

In the following, we give the full description of all Lie algebras of Vergne type over $\z2$ up to dimension $12$. In order to do so in an efficient way, we give only the nonzero relations which involve $e_2$. One can obtain the remaining bracket relations from the Jacobi identity. For an $n$-dimensional Lie algebra of Vergne type, we denote by $[0,c_{2,3},c_{2,4},\dots,c_{2,n-2},0,0]$ the relations $[e_2,e_3]=c_{2,3}e_5,\dots, [e_2,e_{n-2}]=c_{2,n-2}e_n$, where $c_{2,3},\dots,c_{2,n-2} \in \z2$. For example, the algebra $g(8,1)$ with relations $[0, 0, 0, 1, 0, 0, 0]$ is the Lie algebra of Vergne type with relations $[e_1,e_i]=e_{i+1}$, for $2\leq i \leq 7$ and $[e_2,e_5]=e_7, [e_3,e_5]=e_8, [e_3,e_4]=e_7$.

\begin{figure}[H]
    \centering
    \def\svgwidth{\columnwidth}
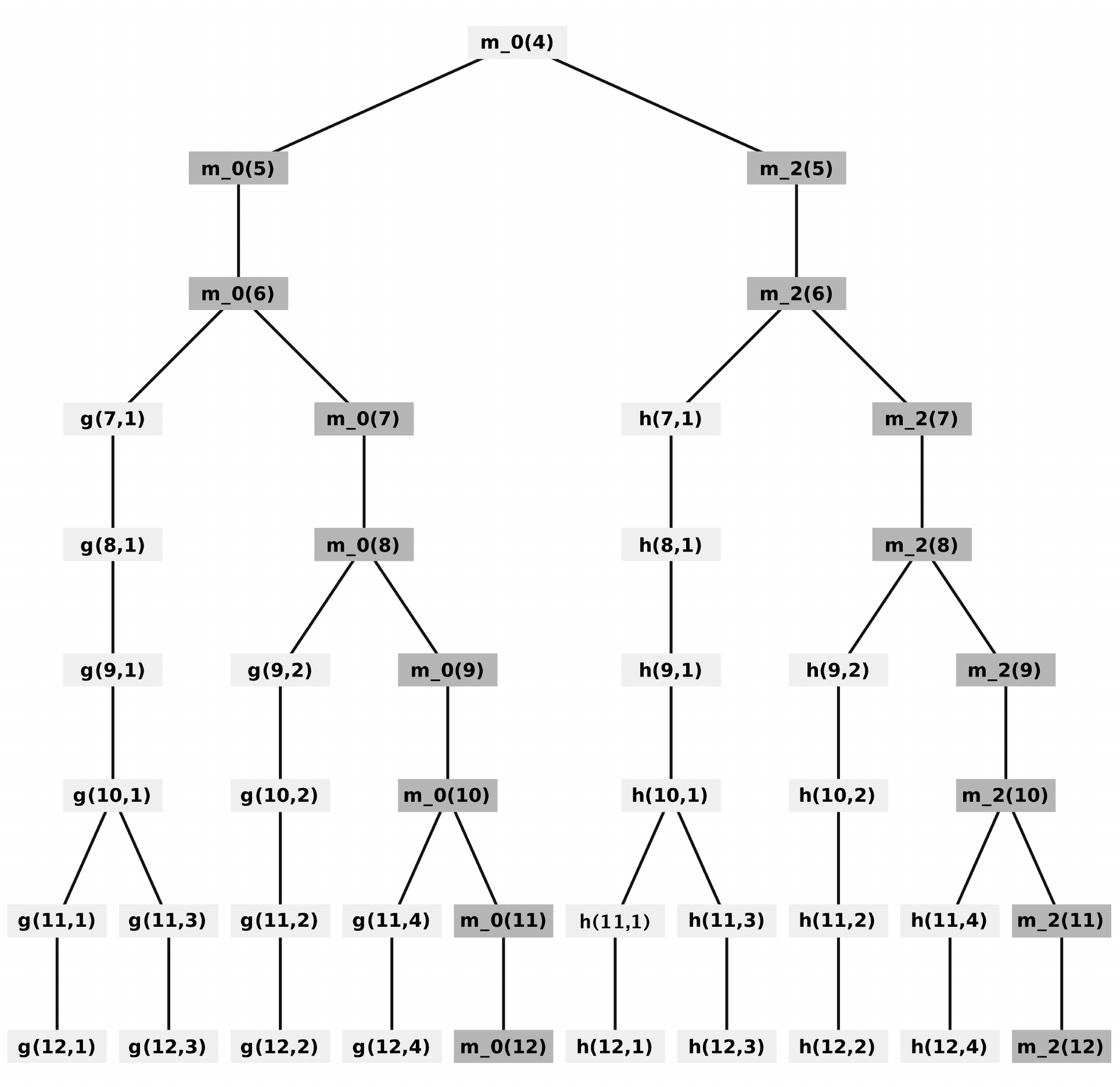
    \caption{Lie algebras of Vergne type of dimension $n \leq 12$.}
    \label{tree}
\end{figure}
\begin{itemize}
\begin{adjmulticols}{2}{0pt}{-30pt}
\item[] $g$(7,1): [0, 0, 0, 1, 0, 0]
\item[] 
\item[] $g$(8,1): [0, 0, 0, 1, 0, 0, 0]
\item[] 
\item[] $g$(9,1): [0, 0, 0, 1, 0, 0, 0, 0]
\item[] $g$(9,2): [0, 0, 0, 0, 0, 1, 0, 0]
\item[] 
\item[] $g$(10,1): [0, 0, 0, 1, 0, 0, 1, 0, 0]
\item[] $g$(10,2): [0, 0, 0, 0, 0, 1, 1, 0, 0]
\item[] 
\item[] $g$(11,1): [0, 0, 0, 1, 0, 0, 1, 0, 0, 0]
\item[] $g$(11,2): [0, 0, 0, 0, 0, 1, 1, 0, 0, 0]
\item[] $g$(11,3): [0, 0, 0, 1, 0, 0, 1, 1, 0, 0]
\item[] $g$(11,4): [0, 0, 0, 0, 0, 0, 0, 1, 0, 0]
\item[] 
\item[] $g$(12,1): [0, 0, 0, 1, 0, 0, 1, 0, 0, 0, 0]
\item[] $g$(12,2): [0, 0, 0, 0, 0, 1, 1, 0, 0, 0, 0]
\item[] $g$(12,3): [0, 0, 0, 1, 0, 0, 1, 1, 0, 0, 0]
\item[] $g$(12,4): [0, 0, 0, 0, 0, 0, 0, 1, 0, 0, 0]
\item[] $h$(7,1): [0, 1, 1, 0, 0, 0]
\item[] 
\item[] $h$(8,1): [0, 1, 1, 0, 1, 0, 0]
\item[] 
\item[] $h$(9,1): [0, 1, 1, 0, 1, 1, 0, 0]
\item[] $h$(9,2): [0, 1, 1, 1, 1, 0, 0, 0]
\item[] 
\item[] $h$(10,1): [0, 1, 1, 0, 1, 1, 0, 0, 0]
\item[] $h$(10,2): [0, 1, 1, 1, 1, 0, 0, 0, 0]
\item[] 
\item[] $h$(11,1): [0, 1, 1, 0, 1, 1, 0, 1, 0, 0]
\item[] $h$(11,2): [0, 1, 1, 1, 1, 0, 0, 1, 0, 0]
\item[] $h$(11,3): [0, 1, 1, 0, 1, 1, 0, 0, 0, 0]
\item[] $h$(11,4): [0, 1, 1, 1, 1, 1, 1, 0, 0, 0]
\item[] 
\item[] $h$(12,1): [0, 1, 1, 0, 1, 1, 0, 1, 1, 0, 0]
\item[] $h$(12,2): [0, 1, 1, 1, 1, 0, 0, 1, 1, 0, 0]
\item[] $h$(12,3): [0, 1, 1, 0, 1, 1, 0, 0, 1, 0, 0]
\item[] $h$(12,4): [0, 1, 1, 1, 1, 1, 1, 0, 1, 0, 0]
\end{adjmulticols}
\end{itemize}

\section*{Acknowledgement}
The author wishes to express his most sincere gratitude to his PhD supervisors, Grant Cairns and Yuri Nikolayevsky, for their inspiring guidance and support.

	\bibliographystyle{amsplain}
	\bibliography{sameleaves}

\end{document}